\documentclass[12pt]{article}
\usepackage{textcomp}
\usepackage{graphicx}           
\usepackage{epstopdf}
\usepackage{color}              
\usepackage{xcolor,graphicx}
\usepackage{geometry}
\usepackage{float}
\usepackage{times}
\usepackage{indentfirst}        
\usepackage{amsmath,amssymb,bm} 
\usepackage{amsthm}
\usepackage{cases}              
\usepackage{setspace}
 
\usepackage{titlesec}
\usepackage[all]{xy}            
\usepackage{tikz}
\usetikzlibrary{arrows.meta}
\usetikzlibrary{shapes.gates.ee}
\usetikzlibrary{calc}
\usepackage{multirow}
\usepackage{graphicx} 
\usepackage{epstopdf}
\usepackage{caption}
\usepackage{diagbox}
\usepackage{mathrsfs}
\usepackage{dsfont}

 \usepackage[colorlinks=true, linkcolor=blue,backref,pagebackref]{hyperref}

\usepackage{tocloft}

\usepackage{tcolorbox}
\usepackage{fancyhdr,fancyvrb}
\usetikzlibrary{shapes.geometric}  

\usepackage{diagbox}
\usepackage{subcaption}
\usepackage{ulem}

\usepackage{booktabs}
\usepackage{makecell}

\newcommand{\E}{\mathbb E}
\newcommand{\Var}{\operatorname{Var}}
\newcommand{\dd}{\,\mathrm d}
\newcommand{\T}{\mathcal T}
\newcommand{\C}{\mathbb C}

\newtheorem{thm}{Theorem}[section]
\newtheorem{defi}[thm]{Definition}
\newtheorem{lem}[thm]{Lemma}

\newtheorem{prop}[thm]{Proposition}

\makeatletter \@addtoreset{equation}{section} \makeatother
\makeindex 

\begin{document}

 \begin{center} 
{\Large \bf On the limiting distribution  of the number of improper\\[5pt]  edges for random  trees  }

\vskip 4mm
Wennie W.J. Ma $^1$ and Kathy Q. Ji $^2$  
\vskip 2mm

 Center for Applied Mathematics, KL-AAGDM\\[5pt]
Tianjin University\\[5pt]
Tianjin 300072, P.R. China \\ \vskip 0.5cm

\vskip 2mm
 Emails: $^1$winniema@tju.edu.cn,   $^2$kathyji@tju.edu.cn 

 \end{center}

\vskip 2mm
\noindent
{\bf Abstract:}     Improper edges were introduced by Shor to refine Cayley's formula for
rooted labeled trees. Zeng established a connection between
Shor's refinement and the Ramanujan polynomials. Let $\mathscr{T}_n$
denote the set of rooted labeled trees on $[n]=\{1,\ldots,n\}$. We prove
that the number of improper edges in a uniformly random tree in
$\mathscr{T}_n$ is asymptotically normal as $n\to\infty$, with mean and
variance asymptotic to $\mu n$ and $\sigma^2 n$, respectively, where
$\mu=e-2$ and $\sigma^2=e^2-3e+1$. This phenomenon was observed by Chen,
and the proof presented here was developed through human--AI
collaboration.

\vskip 2mm



 \vskip 2mm

\section{Introduction}

Motivated by an observation of Chen \cite{Chen-2026}, we establish the
asymptotic normality of the number of improper edges in uniformly random
rooted labeled trees.
Let \(\mathscr{T}_n\) denote the set of rooted labeled trees on the vertex
set \([n]=\{1,2,\ldots,n\}\), equipped with the uniform probability measure.
Throughout this paper, rooted trees are understood to be unordered, so
the children of each vertex are not equipped with a linear order.
We identify each vertex with its label.
For an edge \((i,j)\) of a tree \(T\in\mathscr{T}_n\), we always take
\(i\) to be the parent of \(j\).
The edge \((i,j)\) is called {\it proper} if \(i\) is smaller than the
label of every vertex in the subtree rooted at \(j\), including \(j\)
itself; otherwise, it is called  {\it\textbf{ }improper}.

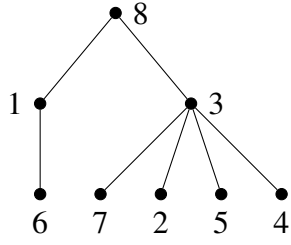
\begin{figure}[H]
\centering
\begin{tikzpicture}[
  scale=0.8,
  vertex/.style={shape=circle, draw, inner sep=1.5pt, fill=black},
  sibling distance=1.5cm,
  level distance=14mm,
  auto
]
\node[vertex,label=0:{8}]{}[
  grow=down, sibling distance=2.5cm, level distance=15mm
]
child {
  node[vertex,label=-180:{1}]{}
  child {
    node[vertex,label=-90:{6}]{}
    edge from parent[solid]
  }
  edge from parent[solid]
}
child {
  node[vertex,label=0:{3}]{}[
    sibling distance=10mm, level distance=15mm
  ]
  child {node[vertex,label=-90:{7}]{} edge from parent[solid]}
  child {node[vertex,label=-90:{2}]{} edge from parent[solid]}
  child {node[vertex,label=-90:{5}]{} edge from parent[solid]}
  child {node[vertex,label=-90:{4}]{} edge from parent[solid]}
  edge from parent[solid]
};
\end{tikzpicture}
\caption{A rooted labeled tree on \([8]\) with three improper edges.}
\label{example}
\end{figure}

For example, the tree \(T\in\mathscr{T}_8\) shown in
Fig.~\ref{example} has exactly three improper edges:
\((8,1)\), \((8,3)\), and \((3,2)\). Its remaining four edges are proper.
Our main result is the following.

\begin{thm}\label{thm:main}
Let \(X_n\) denote the number of improper edges in a uniformly random
rooted labeled tree on \([n]\). As \(n\to\infty\),
\begin{align}
\mathbb{E}X_n
&=\mu n-\frac{e}{2}+O(n^{-1}),\label{eq:mean}\\
\operatorname{Var}(X_n)
&=\sigma^2 n+\frac{e(3-e)}{2}+O(n^{-1}),\label{eq:variance}
\end{align}
where
\[
\mu=e-2,
\qquad
\sigma^2=e^2-3e+1.
\]
Moreover,
\begin{equation}\label{eq:clt}
\frac{X_n-\mu n}{\sigma\sqrt{n}}
\xrightarrow{\mathrm{d}}\mathcal{N}(0,1).
\end{equation}
\end{thm}

Improper edges of rooted labeled trees were introduced by Shor
\cite{Shor-1995} to refine Cayley's formula \cite{Cayley-1889}.
Let \(r(n,k)\) denote the number of rooted labeled trees on \([n]\)
with exactly \(k\) improper edges.
Shor \cite{Shor-1995} proved that, for \(n\ge2\) and
\(0\le k\le n-1\),
\begin{equation}\label{rec-improp}
r(n,k)=(n-1)r(n-1,k)+(n+k-2)r(n-1,k-1),
\end{equation}
with the initial condition \(r(1,0)=1\) and the convention that
\(r(n,k)=0\) for \(k<0\) or \(k\ge n\).   The numbers $r(n,k)$ is listed
as sequences A054589  in OEIS. 
By Cayley's formula,
\[
\sum_{k=0}^{n-1}r(n,k)=n^{n-1}.
\]
For $n\geq 1$, define
\[
R_n(q)=\sum_{k=0}^{n-1}r(n,k)q^k.
\]
Below are the first few values of 
$R_n(q)$:
\begin{align*}
R_1(q)&=1,\\[3pt]
R_2(q)&=1+q,\\[3pt]
R_3(q)&=2+4q+3q^2,\\[3pt]
R_4(q)&=6+18q+25q^2+15q^3,\\[3pt]
R_5(q)&=24+96q+190q^2+210q^3+105q^4,\\[3pt]
R_6(q)&=120+600q+1526q^2+2380q^3+2205q^4+945q^5.
\end{align*}

Zeng \cite{Zeng-1999} identified the connection between Shor's
refinement of Cayley's formula and the Ramanujan polynomials
\cite{Berndt-1985,Ramanujan-1957}.
Dumont and Ramamonjisoa \cite{Dumont-Ramamonjisoa-1996} independently
obtained the interpretation of their coefficients in terms of rooted
labeled trees with a prescribed number of improper edges.
For further details, see Chen, Fu, and Wang
\cite{Chen-Fu-Wang-2026} and Chen and Yang \cite{Chen-Yang-2021}.  
Chen, Wang, and Yang \cite{Chen-Wang-Yang-2011} proved that the
sequence \((R_n(q))_{n\ge1}\) is strongly \(q\)-log-convex; that is,
\[
R_{m-1}(q)R_{n+1}(q)\ge_q R_m(q)R_n(q)
\qquad(n\ge m\ge2),
\]
where \(f(q)\ge_q g(q)\) means that \(f(q)-g(q)\) has nonnegative
coefficients.
Sokal \cite{Sokal-2023} established the stronger property that the
Hankel matrix
\[
\bigl(R_{i+j+1}(q)\bigr)_{i,j\ge0}
\]
is coefficientwise totally positive, meaning that each of its minors
is a polynomial in \(q\) with nonnegative coefficients.

Although the recurrence \eqref{rec-improp} is relatively simple, it
does not immediately yield the limiting distribution of the number
of improper edges. In Section~\ref{sect}, we show that $R_n(q)$ is
not real-rooted for $n\geq 3$, so standard methods for deducing
asymptotic normality from real-rootedness do not apply directly.
Instead, in Section~\ref{sec3}, we prove the main result using an
analytic approach based on singularity analysis and Hwang's
quasi-powers theorem.

\section{Non-real-rootedness of $R_n(q)$} \label{sect}

Lin and Zeng \cite[proof of Corollary~5.1]{Lin-Zeng-2014}
showed that the coefficient sequence of $R_n(q)$ is log-concave
and hence unimodal. The following proposition shows, however,
that $R_n(q)$ is not real-rooted for $n\geq 3$. Thus, these
polynomials provide a family of examples  of log-concavity without real-rootedness.

\begin{prop}
For   $n\ge 3$, the  polynomials $R_n(q)$ are not real‑rooted.
\end{prop} 
\begin{proof}
It follows by induction from \eqref{rec-improp} that, for $n\geq 2$,
\begin{equation}\label{real-aa}
r(n,0)=(n-1)!,
\qquad
r(n,1)=(n-1)(n-1)!.
\end{equation}
For $n\geq 2$, set
\[
b_n=\frac{r(n,2)}{(n-1)!}.
\]
Using \eqref{real-aa}, the recurrence \eqref{rec-improp} gives
\[
b_{n+1}
=
b_n+\frac{n+1}{n}(n-1)
=
b_n+n-\frac1n.
\]
Since $b_2=0$, we obtain
\begin{equation}\label{real-bb}
b_n
=
\sum_{j=1}^{n-1}\left(j-\frac1j\right)
=
\binom{n}{2}-H_{n-1},
\end{equation}
where $H_m=\sum_{j=1}^{m}1/j$ denotes the $m$th harmonic number.

Fix $n\geq 3$ and suppose, for a contradiction, that $R_n(q)$
is real-rooted. Since $R_n(q)$ has degree $n-1$ and strictly
positive coefficients, all its zeros must be negative.
Thus, by \eqref{real-aa}, there exist
$\lambda_1,\ldots,\lambda_{n-1}>0$ such that
\[
\frac{R_n(q)}{(n-1)!}
=
\prod_{j=1}^{n-1}(1+\lambda_jq).
\]
Comparing the coefficients of $q$ and $q^2$ and using
\eqref{real-aa} and \eqref{real-bb}, we obtain
\[
\sum_{j=1}^{n-1}\lambda_j=n-1,
\qquad
\sum_{1\leq i<j\leq n-1}\lambda_i\lambda_j
=
\binom{n}{2}-H_{n-1}.
\]
By the Cauchy--Schwarz inequality,
\[
\sum_{j=1}^{n-1}\lambda_j^2
\geq
\frac{1}{n-1}
\left(\sum_{j=1}^{n-1}\lambda_j\right)^2
=
n-1.
\]
Consequently,
\begin{align*}
\sum_{1\leq i<j\leq n-1}\lambda_i\lambda_j
&=
\frac12\left[
\left(\sum_{j=1}^{n-1}\lambda_j\right)^2
-
\sum_{j=1}^{n-1}\lambda_j^2
\right] \\
&\leq
\frac{(n-1)^2-(n-1)}{2}
=
\binom{n-1}{2}.
\end{align*}
On the other hand, since $H_{n-1}<n-1$ for $n\geq 3$,
\[
\sum_{1\leq i<j\leq n-1}\lambda_i\lambda_j
=
\binom{n}{2}-H_{n-1}
>
\binom{n}{2}-(n-1)
=
\binom{n-1}{2},
\]
a contradiction. Therefore,   $R_n(q)$
is not real-rooted for  $n\geq 3$.
\end{proof}

\section{The limiting distribution} \label{sec3}

Throughout this paper, a labeled tree is understood to be rooted and
unordered, with the vertices of a tree of size \(n\) labeled by the
distinct elements of \([n]=\{1,\ldots,n\}\).
Recall that \(\mathscr{T}_n\) denotes the set of such trees, and write
\(T_n=|\mathscr{T}_n|\).
Cayley's formula \cite{Cayley-1889} gives \(T_n=n^{n-1}\).
The exponential generating function
\[
\T(z):=\sum_{n\ge1}T_n\frac{z^n}{n!},
\]
known as the {\it Cayley tree function}, 
 is characterized via the functional equation
\begin{equation}\label{tree}
\T(z)= z e^{\T(z)},  
\end{equation}
and is thus closely related to the Lambert $W$-function 
\cite{Corless-Gonnet-Hare-Jeffrey-Knuth-1996,Flajolet-Sedgewick-2009}.

Our starting point is the bivariate exponential generating function
\begin{equation}\label{eq:F-def}
F(z,q)=\sum_{n\ge1}R_n(q)\frac{z^n}{n!}.
\end{equation}
We express \(F(z,q)\) in terms of the Cayley tree function
\(\T(z)\). By tracking the dominant singularity of \(F(z,q)\) as \(q\) varies
near \(1\), we obtain a singular expansion that is uniform for \(q\)
in a complex neighborhood of \(1\).
Singularity analysis then yields a quasi-powers expansion for the
probability generating function
\[
\mathbb{E}\bigl[q^{X_n}\bigr]
=\sum_{k=0}^{n-1}\mathbb{P}(X_n=k)q^k
=\frac{R_n(q)}{R_n(1)}.
\]
Hwang's quasi-powers theorem consequently yields the asymptotic
normality of \(X_n\).

We begin by recalling some relevant background. 

\subsection{Preliminaries}

 We first recall the notions of a \(\Delta\)-domain and
\(\Delta\)-analyticity, which are used to describe the singular behavior
of \(\T(z)\).

\begin{defi}{\cite[Definition VI.1, p.~389]{Flajolet-Sedgewick-2009}} Given two numbers $\phi, R$ with $R>1$ and $0<\phi<\frac{\pi}{2}$, the open domain $\Delta(\phi,R)$ is defined as
\[
\Delta(\phi,R) = \{ z \mid |z| < R,\ z \neq 1,\ |\arg(z-1)| > \phi \}.
\]
A domain is a \textit{$\Delta$–domain at $1$} if it is a $\Delta(\phi,R)$ for some $R$ and $\phi$. For a complex number $\zeta \neq 0$, a \textit{$\Delta$–domain at $\zeta$} is the image by the mapping $z \mapsto \zeta z$ of a $\Delta$–domain at $1$. A function is \textit{$\Delta$–analytic} if it is analytic in some $\Delta$–domain.
\end{defi}

The following standard properties of the Cayley tree function will be
used in our analysis; see Flajolet and Sedgewick
\cite[Example~VI.8, pp.~403--404]{Flajolet-Sedgewick-2009}
and Corless et al.
\cite[Secs.~3--4, pp.~339--351]{Corless-Gonnet-Hare-Jeffrey-Knuth-1996}.

\begin{prop}\label{prop:T-properties}
The Cayley tree function \(\T(z)\) has the following properties: 
\begin{enumerate}
\item[(i)]
The radius of convergence of \(\T(z)\) is \(e^{-1}\). Moreover, the Cayley tree function \(\T(z)\) is analytic in the slit plane
\[
\C\setminus[e^{-1},\infty).
\]

\item[(ii)]
As \(z\to e^{-1}\) in a \(\Delta\)-domain at \(e^{-1}\), the following Puiseux expansion holds:
\begin{equation}\label{treefunctioexpan}
\T(z)
=
1-\sqrt{2(1-ez)}
+\frac{2}{3}(1-ez)
+
O\left((1-ez)^{3/2}\right),
\end{equation}
where the branch of \((1-ez)^{1/2}\) is chosen by analytic continuation from the interval \((0,e^{-1})\), so that
\[
(1-ez)^{1/2}>0
\]
for real \(z\in(0,e^{-1})\).

\item[(iii)]
If \(z\ne0\) and \(\T(z)\ne1\), then
\begin{equation}\label{eq:T-derivative}
\T'(z)=\frac{\T(z)}{z(1-\T(z))}.
\end{equation}
\end{enumerate}
\end{prop} 

These analytic properties also play a role in classical limit theorems
for statistics on rooted labeled trees, see, for example, the root degree
and the number of leaves
\cite[Examples~IX.6 and~IX.25]{Flajolet-Sedgewick-2009},
fixed-degree vertex counts
\cite[Theorem~2.3 and the subsequent remark]{Drmota-Gittenberger-1999},
and occurrences of fixed tree patterns with matching internal vertex degrees
\cite[Theorem~1.1]{Chyzak-Drmota-Klausner-Kok-2008}.

The two theorems below enable us to derive the asymptotic expansion of the normalized probability generating function $\mathbb{E}(q^{X_n})$ based on the asymptotic behavior of the bivariate exponential generating function $F(z,q)$.

\begin{thm}{\cite[Theorem VI.1 (Standard function scale), p.~381]{Flajolet-Sedgewick-2009}}\label{Standard function scale}
Let $\alpha$ be an arbitrary complex number in $\mathbb{C}\setminus\mathbb{Z}_{\le 0}$.
The coefficient of $z^n$ in
\[
f(z) = (1-z)^{-\alpha}
\]
admits for large $n$ a complete asymptotic expansion in descending powers of $n$,
\[
[z^n]f(z) \sim \frac{n^{\alpha-1}}{\Gamma(\alpha)}
\left(1+\sum_{k=1}^{\infty}\frac{e_k}{n^k}\right),
\]
where $e_k$ is a polynomial in $\alpha$ of degree $2k$. In particular:
\begin{equation}\label{eq:VI.1-13}
\begin{aligned}
[z^n]f(z) \sim \frac{n^{\alpha-1}}{\Gamma(\alpha)}
\bigg(&1+\frac{\alpha(\alpha-1)}{2n}
+\frac{\alpha(\alpha-1)(\alpha-2)(3\alpha-1)}{24n^2}\\
&+\frac{\alpha^2(\alpha-1)^2(\alpha-2)(\alpha-3)}{48n^3}
+O\left(\frac{1}{n^4}\right)\bigg).
\end{aligned}
\end{equation}
\end{thm}

\begin{thm}{\cite[Theorem VI.3 (Transfer, Big-Oh and little-oh), p.~390]{Flajolet-Sedgewick-2009}} 
\label{thm:transfer-big-o}
Let \(\alpha,\beta\) be arbitrary real numbers, and let \(f(z)\) be a
function that is \(\Delta\)-analytic.

\begin{enumerate}
\item Assume that \(f(z)\) satisfies, in the intersection of a neighbourhood
of \(1\) with its \(\Delta\)-domain, the condition
\[
f(z)
=
O\left(
(1-z)^{-\alpha}
\left(
\log\frac{1}{1-z}
\right)^\beta
\right).
\]
Then
\[
[z^n]f(z)
=
O\left(
n^{\alpha-1}(\log n)^\beta
\right).
\]

\item Assume that \(f(z)\) satisfies, in the intersection of a neighbourhood
of \(1\) with its \(\Delta\)-domain, the condition
\[
f(z)
=
o\left(
(1-z)^{-\alpha}
\left(
\log\frac{1}{1-z}
\right)^\beta
\right).
\]
Then
\[
[z^n]f(z)
=
o\left(
n^{\alpha-1}(\log n)^\beta
\right).
\]
\end{enumerate}
\end{thm}

Finally, we recall Hwang's quasi-powers theorem, which is the main
probabilistic tool in our proof.
We use the formulation in
\cite[Lemma~IX.1, p.~646]{Flajolet-Sedgewick-2009}; see also
Hwang \cite{Hwang-1996} for related results.

\begin{thm}[Quasi-powers, general distributions]\label{lem:hwang}
 Assume that the Laplace transforms $\lambda_n(s) = \mathbb{E}\{e^{sX_n}\}$ of a sequence of random variables $X_n$ are analytic in a disc $|s| < \rho$, for some $\rho > 0$, and satisfy there is an expansion of the form
\begin{equation}
\lambda_n(s) = e^{\beta_n U(s) + V(s)} \left(1 + O\left(\frac{1}{\kappa_n}\right)\right),
\end{equation}
with $\beta_n, \kappa_n \to +\infty$ as $n \to +\infty$, and $U(s), V(s)$ analytic in $|s| \le \rho$. Assume also the variability condition, $U''(0) \neq 0$. Under these assumptions, the mean and variance of $X_n$ satisfy
\begin{align}
\mathbb{E}(X_n) &= \beta_n U'(0) + V'(0) + O(\kappa_n^{-1}), \\
\mathbb{V}(X_n) &= \beta_n U''(0) + V''(0) + O(\kappa_n^{-1}).
\end{align}
The distribution of $X_n^* := \frac{X_n - \beta_n U'(0)}{\sqrt{\beta_n U''(0)}}$ is asymptotically Gaussian, the speed of convergence to the Gaussian limit being $O(\kappa_n^{-1} + \beta_n^{-1/2})$.
\end{thm}

\subsection{Proof of Theorem \ref{thm:main}}

To apply Hwang's quasi-powers theorem and prove Theorem \ref{thm:main}, the key step is to obtain an asymptotic expansion for the normalized probability generating function $\mathbb{E}(q^{X_n})$ of $X_n$. To this end, we first consider the bivariate exponential generating function $F(z,q)$ given by \eqref{eq:F-def}. Our goal is to express $F(z,q)$ explicitly in terms of the Cayley tree function and then analyze its dominant singularity uniformly in $q$.

Josuat-Vergès \cite[Theorem 4.2]{Josuat-Verges-2015} showed that
\begin{equation}\label{eq:F-explicitaa}
\sum_{n\ge 0} G_n(x)\frac{z^n}{n!}
= \T\left( \frac{z+x}{1+x}\exp\left(-\frac{x}{1+x}\right) \right),
\end{equation}
where  \(G_0(x-1)=\frac{x-1}{x}=1-\frac{1}{x}\) and \(G_n(x-1)=R_n(x)\) for \(n\geq 1\).

Setting \(x=q-1\) in \eqref{eq:F-explicitaa} and shifting the constant term to the right-hand side yields
\begin{equation}\label{eq:F-explicit}
F(z,q)=\sum_{n\ge 1} R_n(q)\frac{z^n}{n!}
= \frac1q-1+\T\left( e^{1/q-1}\frac{z+q-1}{q}\right).
\end{equation}

The following uniform \(\Delta\)-domains of $F(z,q)$  are needed in order to apply the transfer
theorem  uniformly with respect to the parameter \(q\).

\begin{lem}\label{lem:unique-dominant-singularity}
There exists \(\varepsilon>0\) such that, whenever \(|q-1|\le \varepsilon\), the function \(F(z,q)\), as a function of \(z\), has the unique dominant singularity \(z=\rho(q)\). Moreover, one can choose a family of \(\Delta\)-domains, uniformly in \(q\),
\begin{equation}\label{defi:deltaq}
\Delta_q
=
\left\{
z \mid \ |z|<|\rho(q)|+\eta,\ z \neq \rho(q),\ \left|\arg\left(\frac{z}{\rho(q)}-1\right)\right|>\theta
\right\},
\end{equation}
where \(\eta>0\) and \(0<\theta<\pi/2\) are independent of \(q\), such that \(F(z,q)\) is analytic in \(\Delta_q\) .
\end{lem}

\begin{proof}
Set
\begin{equation}\label{3.6}
    w=\omega(z,q):=A(q)z+B(q),
\end{equation}
where
\[
A(q)=\frac{e^{1/q-1}}{q},
\qquad
B(q)=\frac{e^{1/q-1}(q-1)}{q}.
\]
Both \(A(q)\) and \(B(q)\) are analytic in a neighborhood of \(q=1\).
Since \(A(1)=1\), continuity of \(A(q)\) at \(q=1\) implies that \(A(q)\ne0\) for all \(q\) in a sufficiently small complex neighbourhood of \(1\). By Proposition \ref{prop:T-properties}  (i) and (ii), the principal branch of the Cayley tree function \(\T(z)\) is analytic in
\(\C\setminus[e^{-1},\infty)\) and its dominant singularity is the square-root branch point \(z=e^{-1}\). From the explicit representation \eqref{eq:F-explicit}, the possible obstructions to analytic continuation of the chosen branch of \(F(z,q)\) are contained in the preimage of the slit \([e^{-1},\infty)\) under the affine map \(z\mapsto\omega(z,q)\).

The preimage of the endpoint \(w=e^{-1}\) is determined by
\[
A(q)z+B(q)=e^{-1}.
\]
Hence
\[
z=\rho(q):=\frac{e^{-1}-B(q)}{A(q)}
=
1-q+qe^{-1/q}.
\]
Evidently, \(\rho(q)\) is analytic in a neighbourhood of \(q=1\) and \(\rho(1)=e^{-1}>0\)  .

Next we describe the whole preimage of the slit. Set 
\(
w(z, q)=e^{-1}+\lambda\) with \(\lambda\ge0\).
By \eqref{3.6}, we have 
\[
A(q)z+B(q)=e^{-1}+\lambda.
\]
Hence, for $\lambda\ge0$,  we have \[
z=\rho(q)+\lambda d(q), \qquad d(q)=\frac1{A(q)}.
\]
Thus, the slit in the \(w\)-plane pulls back to the ray
\[
L_q=
\left\{
\rho(q)+\lambda d(q):\lambda\ge0
\right\}.
\]

At \(q=1\), we have
\(
\rho(1)=e^{-1}, d(1)=1.
\)
Therefore
\[
\operatorname{Re}\bigl(\overline{\rho(1)}d(1)\bigr)=e^{-1}>0.
\]
By continuity of the function
$\operatorname{Re}\!\left(\overline{\rho(q)}\,d(q)\right)$
at \(q=1\), there exists $\varepsilon>0$ such that
\(
\operatorname{Re}\bigl(\overline{\rho(q)}d(q)\bigr)>0
\)
for all \(|q-1|\le\varepsilon\). Consequently, at this moment, for every \(\lambda>0\),
\[
\begin{aligned}
|\rho(q)+\lambda d(q)|^2
&=
|\rho(q)|^2
+
2\lambda\operatorname{Re}\bigl(\overline{\rho(q)}d(q)\bigr)
+
\lambda^2|d(q)|^2\\
&>
|\rho(q)|^2.
\end{aligned}
\]
Thus, the endpoint \(z=\rho(q)\) is the unique point of minimal modulus on \(L_q\). Since $A(q)\neq 0$, the representation \eqref{eq:F-explicit} and
Proposition \ref{prop:T-properties}(ii) imply that $z=\rho(q)$ is a
square-root branch point of $F(z,q)$. Hence \(z=\rho(q)\) is the unique dominant singularity of \(F(z,q)\).

It remains to justify that $F(z,q)$ is analytic in $\Delta_q$ given by \eqref{defi:deltaq}. Since
\[
\frac{d(1)}{\rho(1)}=e>0,
\]
by continuity of the function  ${d(q)/\rho(q)}$ at $q=1$, there exists  \(\varepsilon>0\) and \(0<\theta<\pi/2\) such that
\(
\left|\arg\frac{d(q)}{\rho(q)}\right|\leq\theta
\)
for all \(|q-1|\le\varepsilon\). Therefore \(\forall z \in L_q \setminus \{\rho(q)\}\),
\[
\left|\arg\left(\frac{z}{\rho(q)}-1\right)\right|=\left|\arg\lambda\frac{d(q)}{\rho(q)}\right|=\left|\arg\frac{d(q)}{\rho(q)}\right|\le\theta.
\]
Since \(q\) ranges over a sufficiently small compact neighbourhood of \(1\), the constants \(\eta\) and \(\theta\) can be chosen uniformly. 
Since $\rho(q)\notin\Delta_q$, it follows that
$\Delta_q\cap L_q=\varnothing$.
Hence 
$F(z,q)$ is analytic in $\Delta_q$.
This completes the proof.
\end{proof}

Define
\begin{equation}\label{eq:rho-kappa}
\rho(q)=1-q+qe^{-1/q},
\qquad
\kappa(q)=\frac{e^{1/q}\rho(q)}{q}.
\end{equation}
Then \(\rho(1)=e^{-1}\) and \(\kappa(1)=1\).

\begin{lem}\label{lem:uniform-coeff}
There exists \(\varepsilon>0\) such that, as \(n\to\infty\), uniformly for \(|q-1|\le\varepsilon\),
\begin{equation}\label{eq:coeff-asympt}
R_n(q)
=
\frac{n!}{\sqrt{2\pi}}\sqrt{\kappa(q)}\,
\rho(q)^{-n}n^{-3/2}
\left(1+O(n^{-1})\right),
\end{equation}
where $\rho(q)$ and $\kappa(q)$ are given in \eqref{eq:rho-kappa}.
\end{lem}

\begin{proof}  Under the affine substitution adopted in \eqref{3.6}, we have
\begin{equation}\label{eq:linear}
1-ew(z ,q)
=1-e(e^{1/q-1}\frac{z+q-1}{q})=
\kappa(q)\left(1-\frac z{\rho(q)}\right).
\end{equation}
Substituting this relation into \eqref{treefunctioexpan} from Proposition \ref{prop:T-properties}, together with the identity \eqref{eq:F-explicit}, we further deduce that
\[
F(z,q)
=
H_1(z,q)
-
\sqrt{2\kappa(q)}
\left(1-\frac{z}{\rho(q)}\right)^{1/2}
+
O\left(
\left(1-\frac{z}{\rho(q)}\right)^{3/2}
\right),
\]
where
\[
H_1(z,q)
=
\frac1q+
\frac23\kappa(q)
\left(1-\frac{z}{\rho(q)}\right).
\]
In particular, \(H_1(z,q)\) is analytic in \(z\) at
\(z=\rho(q)\). By Lemma~\ref{lem:unique-dominant-singularity}, this expansion holds in a family of \(\Delta\)-domains uniformly in \(q\).

For \(n\ge2\), the analytic term \(H_1(z,q)\) has zero \(z^n\)-coefficient.
Hence only the singular term and the remainder have to be considered.

By Theorem~\ref{Standard function scale}, we have
\[
\begin{aligned}
[z^n]\left(1-\frac{z}{\rho(q)}\right)^{1/2}
&=
\frac{\rho(q)^{-n}n^{-3/2}}{\Gamma(-1/2)}
\left(
1+\frac{3}{8n}+O(n^{-2})
\right) \\
&=
\frac{\rho(q)^{-n}n^{-3/2}}{\Gamma(-1/2)}
\left(1+O(n^{-1})\right).
\end{aligned}
\]
On the other hand, by Theorem
\ref{thm:transfer-big-o},
\[
[z^n]
O\left(
\left(1-\frac{z}{\rho(q)}\right)^{3/2}
\right)
=
O\left(
\rho(q)^{-n}n^{-5/2}
\right).
\]

Since
\[
\frac{\rho(q)^{-n}n^{-5/2}}
{\rho(q)^{-n}n^{-3/2}}
=
n^{-1},
\]
the remainder is of relative order \(O(n^{-1})\) with respect to the
leading term.

By Lemma~\ref{lem:unique-dominant-singularity}, the corresponding \(\Delta\)-domains and the constants in
the local expansions can be chosen uniformly for \(q\) near \(1\).
Consequently, the above coefficient estimates are uniform in \(q\).
As \(n\to\infty\), we thus obtain
\begin{equation} \label{eq:coeff-asymptaa}
[z^n]F(z,q)
=
-\frac{\sqrt{2\kappa(q)}}{\Gamma(-1/2)}
\rho(q)^{-n}n^{-3/2}
\left(1+O(n^{-1})\right),
\end{equation}
uniformly for \(q\) in a neighbourhood of \(1\).

Substituting 
\[
\Gamma(-1/2)=-2\sqrt{\pi}, \quad  R_n(q)=n![z^n]F(z,q),
\]
into \eqref{eq:coeff-asymptaa} yields the asymptotic estimate \eqref{eq:coeff-asympt}. This completes the proof.
\end{proof}

We are now in a position to derive the asymptotic expansion of the normalized probability generating function \(\mathbb{E}(q^{X_n})\) by virtue of Lemma \ref{lem:uniform-coeff}.

\begin{thm} \label{thm:quasi} Let $X_n$ denote the number of improper edges in a
uniformly random labeled rooted  tree on $[n]$. 
There exists \(\delta>0\) such that, uniformly for \(|s|\le\delta\), as \(n\to\infty\),
\begin{equation}\label{ }
\mathbb E(e^{sX_n})
=
\sqrt{\kappa(e^s)}
\left(
\frac{\rho(1)}{\rho(e^s)}
\right)^n
\left(1+O(n^{-1})\right),
\end{equation}
where $\rho(q)$ and $\kappa(q)$ are given in \eqref{eq:rho-kappa}.

\end{thm}
\begin{proof}
The probability generating function of $X_n$ is
\[
\mathbb{E}(q^{X_n})
=
\sum_{k=0}^{n-1}q^k\mathbb{P}(X_n=k)
=
\frac{R_n(q)}{R_n(1)}.
\]
Setting $q=e^s$, we obtain the moment generating function
\[
M_n(s)
:=
\mathbb{E}(e^{sX_n})
=
\frac{R_n(e^s)}{R_n(1)}.
\]

By Lemma~\ref{lem:uniform-coeff}, there exist constants
$\varepsilon,C>0$ such that, for all sufficiently large $n$,
\[
R_n(q)
=
\frac{n!}{\sqrt{2\pi}}
\sqrt{\kappa(q)}\,\rho(q)^{-n}n^{-3/2}
\bigl(1+\xi_n(q)\bigr),
\qquad |q-1|\leq\varepsilon,
\]
where the remainder satisfies
\[
\sup_{|q-1|\leq\varepsilon}|\xi_n(q)|
\leq \frac{C}{n}.
\]
Choose $\delta>0$ sufficiently small that
$|e^s-1|\leq\varepsilon$ whenever $|s|\leq\delta$.
For all sufficiently large $n$, we have $C/n\leq 1/2$,
and therefore
\[
|1+\xi_n(1)|
\geq 1-|\xi_n(1)|
\geq 1-\frac{C}{n}
\geq \frac12.
\]

Since $\kappa(1)=1$, we obtain
\begin{align*}
M_n(s)
&=
\frac{R_n(e^s)}{R_n(1)}\\
&=
\sqrt{\kappa(e^s)}
\left(\frac{\rho(1)}{\rho(e^s)}\right)^n
\frac{1+\xi_n(e^s)}{1+\xi_n(1)}\\
&=
\sqrt{\kappa(e^s)}
\left(\frac{\rho(1)}{\rho(e^s)}\right)^n
\left(
1+\frac{\xi_n(e^s)-\xi_n(1)}{1+\xi_n(1)}
\right).
\end{align*}
Moreover, uniformly for $|s|\leq\delta$,
\[
\left|
\frac{\xi_n(e^s)-\xi_n(1)}{1+\xi_n(1)}
\right|
\leq
\frac{|\xi_n(e^s)|+|\xi_n(1)|}{|1+\xi_n(1)|}
\leq
\frac{2C/n}{1-C/n}
\leq \frac{4C}{n}.
\]
Consequently, as $n\to\infty$,
\[
M_n(s)
=
\sqrt{\kappa(e^s)}
\left(\frac{\rho(1)}{\rho(e^s)}\right)^n
\bigl(1+O(n^{-1})\bigr),
\]
uniformly for $|s|\leq\delta$.
\end{proof}

We conclude the paper by proving our main result via Theorem \ref{thm:quasi} and Hwang’s quasi-powers theorem \ref{lem:hwang}.

\begin{proof}[Proof of Theorem \ref{thm:main}] 
By Theorem~\ref{thm:quasi}, the moment generating function of \(X_n\) satisfies,
uniformly for \(s\) in a complex neighbourhood of \(0\),
\begin{equation}\label{eq:main-proof-quasi}
M_n(s):=\E( e^{sX_n})
=
\exp\bigl(nU(s)+V(s)\bigr)
\left(1+O(n^{-1})\right),
\end{equation}
where
\[
U(s)=\log \frac{\rho(1)}{\rho(e^s)},
\qquad
V(s)=\frac12\log \kappa(e^s).
\]
Moreover,
\[
U(0)=V(0)=0.
\]
Thus the hypotheses of Theorem~\ref{lem:hwang} will be satisfied once we verify
that \(U''(0)>0\).

We now compute the relevant derivatives. Since
\(
\rho(q)=1-q+qe^{-1/q},
\)
we have
\[
\rho'(q)
=
-1+e^{-1/q}\left(1+\frac1q\right),
\qquad
\rho''(q)
=
\frac{e^{-1/q}}{q^3}.
\]
Therefore
\[
\rho(1)=e^{-1},
\qquad
\frac{\rho'(1)}{\rho(1)}=2-e,
\qquad
\frac{\rho''(1)}{\rho(1)}=1.
\]

Let
\(
D=q\frac{\dd}{\dd q}.
\)
Under the substitution \(q=e^s\), differentiation with respect to \(s\)
corresponds to applying the operator \(D\). Hence
\[
U'(0)
=
-D\log\rho(q)\big|_{q=1}
=
-\frac{\rho'(1)}{\rho(1)}
=
e-2
=
\mu.
\]
Furthermore,
\[
\begin{aligned}
U''(0)
&=
-D^2\log\rho(q)\big|_{q=1}  \\
&=
-\left(
\frac{\rho'(1)}{\rho(1)}
+
\frac{\rho''(1)}{\rho(1)}
-
\left(\frac{\rho'(1)}{\rho(1)}\right)^2
\right)  \\
&=
e^2-3e+1
=
\sigma^2.
\end{aligned}
\]
Also,
\[
\sigma^2=e^2-3e+1>0.
\]

Next, since
 \(
\kappa(q)=\frac{e^{1/q}\rho(q)}{q},\)
we have
\[
\log\kappa(q)=\frac1q+\log\rho(q)-\log q.
\]
It follows that
\[
V(s)
=
\frac12\log\kappa(e^s)
=
\frac12\left(e^{-s}+\log\rho(e^s)-s\right).
\]
Using the values computed above, we obtain
\[
V'(0)
=
\frac12\left(-1+\frac{\rho'(1)}{\rho(1)}-1\right)
=
-\frac e2.
\]
Moreover,
\[
V''(0)
=
\frac12\left(
1+D^2\log\rho(q)\big|_{q=1}
\right).
\]
Since
\[
D^2\log\rho(q)\big|_{q=1}
=
-U''(0)
=
-(e^2-3e+1),
\]
we get
\[
V''(0)
=
\frac12\left(1-e^2+3e-1\right)
=
\frac{e(3-e)}2.
\]
From Theorem ~\ref{lem:hwang}, we have 
\[
\E X_n
=
\mu n-\frac e2+O(n^{-1}),
\]
and
\[
\Var(X_n)
=
\sigma^2 n+\frac{e(3-e)}2+O(n^{-1}).
\]
Moreover, 
\[
\frac{X_n-\mu n}{\sigma\sqrt n}
\xrightarrow{\mathrm d} N(0,1).
\]
This completes the proof.
 
\end{proof}

 \vskip 0.2cm
\noindent{\bf Acknowledgment.} This work
was supported by   the National Science Foundation of China.


\begin{thebibliography}{99}





\bibitem{Berndt-1985}B.~C. Berndt, {\it Ramanujan's notebooks. Part I}, Springer, New York, 1985; MR0781125.

\bibitem{Cayley-1889}
A. Cayley,
A theorem on trees,
Quart. J. Math. 23 (1889), 376--378.

\bibitem{Chen-2026} W.~Y.~C. Chen, Some observations and questions via Maple, talk at AlCoVE: an
Algebraic Combinatorics Virtual Expedition, virtual conference, June 8–9, 2026.
 
  \bibitem{Chen-Fu-Wang-2026} W.~Y.~C. Chen, A.~M. Fu and E.~L. Wang, A grammatical calculus for the Ramanujan polynomials, Ramanujan J., 70 (2026) 35.

 \bibitem{Chen-Yang-2021} W.~Y.~C. Chen, H.~R.~L. Yang, A context-free grammar for the Ramanujan-Shor polynomials, Adv. in Appl. Math., 126 (2021) 101908.


 \bibitem{Chen-Wang-Yang-2011} W.~Y.~C. Chen, L.~X.~W. Wang and A.~L.~B. Yang, Recurrence relations for strongly $q$-log-convex polynomials, Canad. Math. Bull., 54 (2011) 217–229.

  

\bibitem{Corless-Gonnet-Hare-Jeffrey-Knuth-1996}
R.~M. Corless et al., On the Lambert $W$ function, Adv. Comput. Math. {\bf 5} (1996), no.~4, 329--359; MR1414285.


\bibitem{Chyzak-Drmota-Klausner-Kok-2008}F. Chyzak et al., The distribution of patterns in random trees, Combin. Probab. Comput. {\bf 17} (2008), no.~1, 21--59; MR2376422.


\bibitem{Drmota-Gittenberger-1999}M. Drmota and B. Gittenberger, The distribution of nodes of given degree in random trees, J. Graph Theory {\bf 31} (1999), no.~3, 227--253; MR1688949.

\bibitem{Dumont-Ramamonjisoa-1996}D. Dumont and A. Ramamonjisoa, Grammaire de Ramanujan et arbres de Cayley, Electron. J. Combin. {\bf 3} (1996), no.~2, Research Paper 17, approx.\ 18 pp.; MR1392502.





\bibitem{Flajolet-Sedgewick-2009}P. Flajolet and R. Sedgewick, {\it Analytic combinatorics}, Cambridge Univ. Press, Cambridge, 2009; MR2483235.



\bibitem{Hwang-1996}H.~K. Hwang, Large deviations for combinatorial distributions. I. Central limit theorems, Annals of Applied Probability 6 (1996), 297–319. 



\bibitem{Josuat-Verges-2015}M. Josuat-Verg\`es, Derivatives of the tree function, Ramanujan J. {\bf 38} (2015), no.~1, 1--15; MR3396483.


\bibitem{Lin-Zeng-2014}
Z. Lin and J. Zeng,
Positivity properties of Jacobi--Stirling numbers and generalized
Ramanujan polynomials,
Adv. in Appl. Math. {\bf 53} (2014), 12--27.






	\bibitem{Ramanujan-1957} 
S. Ramanujan, \textit{Notebooks}, Vol.~1, Tata Institute of Fundamental Research, Bombay, 1957.


\bibitem{Sokal-2023}
A.~D. Sokal, Total positivity of some polynomial matrices that enumerate labeled trees and forests I: forests of rooted labeled trees, Monatsh. Math. {\bf 200} (2023), no.~2, 389--452; MR4544303.


\bibitem{Shor-1995}P.~W. Shor, A new proof of Cayley’s formula for counting labeled trees, Journal of Combinatorial Theory, Series A 71 (1995), 154–158. 










\bibitem{Zeng-1999} J. Zeng, A Ramanujan sequence that refines the Cayley formula for trees, Ramanujan J. 3 (1999), 45–54.


\end{thebibliography}
\end{document}